\documentclass[a4paper,12pt]{article}

\usepackage[applemac]{inputenc}
\usepackage{enumerate}
\usepackage{lmodern}
\usepackage[T1]{fontenc}
\usepackage{verbatim}
\usepackage{textcomp}
\usepackage[english]{babel}
\usepackage[a4paper,vmargin={3.5cm,3.5cm},hmargin={2.5cm,2.5cm}]{geometry}
\usepackage[font=sf, labelfont={sf,bf}, margin=1cm]{caption}
\usepackage[pdftex]{hyperref}
\usepackage[pdftex]{color,graphicx}

\usepackage{amsmath,amsfonts,amssymb,amsthm,mathrsfs}

\usepackage{colonequals} 

\usepackage{titling}  
\thanksmarkseries{alph}

\renewcommand{\theequation}{\thesection.\arabic{equation}}
\newtheorem{theorem}{Theorem}[section]
\newtheorem{lemma}[theorem]{Lemma}

\newcommand{\eqnsection}{
\renewcommand{\theequation}{\thesection.\arabic{equation}}
    \makeatletter
    \csname  @addtoreset\endcsname{equation}{section}
    \makeatother}
\eqnsection

\def\T{{\mathbb T}}
\def\R{{\mathbb R}}
\def\P{{\mathbb P}}
\def\E{{\mathbb E}}
\def\var{\operatorname{Var}}

\def\Z{{\mathbb Z}}

\def\build#1_#2^#3{\mathrel{\mathop{\kern 0pt#1}\limits_{#2}^{#3}}}
\def\rem{\noindent{\bf Remark. }}

\def\ind{{\bf 1}_}

\title{Resistance growth of branching random networks}

\author{
Dayue Chen\thanks{Department of Probability and Statistics, School of Mathematical Sciences, Peking University, Beijing, China, \textit{E-mail}: \texttt{dayue@pku.edu.cn}}, 
\hspace{1mm}
Yueyun Hu\thanks{LAGA, Universit\'e Paris XIII, Villetaneuse, France, \textit{E-mail}: \texttt{yueyun@math.univ-paris13.fr} } 
\hspace{1mm} 
and 
Shen Lin\thanks{Sorbonne Universit\'e, Laboratoire de Probabilit\'es Statistique et Mod\'elisation, Paris, France, \textit{E-mail}: \texttt{shen.lin.math@gmail.com}
\newline Cooperation between D.C.~and Y.H.~was supported by NSFC 11528101, Research of S.L. was partially supported by the grant ANR-14-CE25-0014 (ANR GRAAL).}
}

\date{\tiny\today}

\begin{document}

\maketitle 

\begin{abstract}

Consider a rooted infinite Galton--Watson tree with mean offspring number $m>1$, and a collection of i.i.d.~positive random variables $\xi_e$ indexed by all the edges in the tree. We assign the resistance $m^d\,\xi_e$ to each edge $e$ at distance $d$ from the root. In this random electric network, we study the asymptotic behavior of the effective resistance and conductance between the root and the vertices at depth $n$.  Our results generalize an existing work of Addario-Berry, Broutin and Lugosi on the binary tree to random branching networks. 


\smallskip
\noindent{\bf AMS 2010 Classification Numbers.}  60F05, 60J80. 
\end{abstract}

\section{Introduction}

An electric network is an undirected locally finite connected graph $G=(V,E)$ with a countable set of vertices $V$ and  a set of edges $E$, endowed with nonnegative numbers $\{r(e), e\in E\}$, called resistances, that are associated to the edges of $G$. The reciprocal $c(e)=1/r(e)$ is called the conductance of the edge $e$. 
It is well-known that the electrical properties of the network $(G,\{r(e)\})$ are closely related to the nearest-neighbor random walk on $G$, whose transition probabilities from a vertex are proportional to the conductances along the edges to be taken. See, for instance, the book of Lyons and Peres \cite{LP} for a detailed exposition of this connection. 

To study random walks in certain random environments, it is natural to consider a random electric network by choosing the resistances independent and identically distributed. 
For example, the infinite cluster of bond percolation on $\Z^d$ can be seen as a random electric network in which each open edge has unit resistance and each closed edge has infinite resistance. Grimmett, Kesten and Zhang \cite{GKZ} proved that when $d\geq 3$, the effective resistance of this network between a fixed point and infinity is a.s.~finite, thus the simple random walk on this infinite percolation cluster is a.s.~transient.
In \cite{BR}, Benjamini and Rossignol considered a different model of the cubic lattice $\Z^d$, where the resistance of each edge is an independent copy of a Bernoulli random variable. They showed that point-to-point effective resistance has submean variance in $\Z^2$, whereas the mean and the variance are of the same order when $d\geq 3$. 
The case of a complete graph on $n$ vertices has also been studied by Grimmett and Kesten \cite{GK1}. For a particular class of resistance distribution on the edges (see Theorem 3 in \cite{GK1}), as $n\to \infty$, the limit distribution of the random effective resistance between two specified vertices was identified as the sum of two i.i.d.~random variables, each with the distribution of the effective resistance between the root and infinity in a Galton--Watson tree with a supercritical Poisson offspring distribution. 

In this paper, we investigate the effective resistance and conductance in a supercritical Galton--Watson tree $\T$ rooted at $\varnothing$. Let $\textbf{p}=(p_k)_{k\geq 0}$ be the offspring distribution of $\T$, with finite mean $m>1$. 
We assume $p_0=0$ to avoid the conditioning on survival. 
Formally, every vertex in $\T$ can be represented as a finite word written with positive integers. 
The depth $|x|$ of a vertex $x$ in $\T$ is the number of edges on the unique non-self-intersecting path from the root $\varnothing$ to $x$, which also equals the length of the word representing $x$. 
Let $\T_n\colonequals \{x\in \T\colon |x|=n\}$ denote the $n$-th level of $\T$. 
We write $\overleftarrow{x}$ for the parent vertex of $x$ if $x\neq \varnothing$.
For each edge $e=\{\overleftarrow{x}, x\}$ of $\T$, we define its depth $d(e)\colonequals |x|$. Let $\nu$ be the number of children of the root, whose expected value is $m$. For $1\leq i\leq \nu$, the edge $\{\varnothing, i\}$ between the root $\varnothing$ and its child $i$ has depth $1$.
If $x$ and $y$ are vertices of $\T$, we write $x \preceq y$ if $x$ is on the non-self-intersecting path connecting $\varnothing$ and $y$. In this case, we say that $y$ is a descendant of $x$. We define $\T_n[x]\colonequals \{y\in \T_n \colon x \preceq y \}$ as the set of vertices at depth $n$ that are descendants of $x$.

If the resistance of an edge at depth $d$ equals $\lambda^{d}$ with a deterministic $\lambda>0$, Lyons \cite{Lyons1} showed that the effective resistance between the root and infinity in $\T$ is a.s.~infinite if $\lambda>m$ and a.s.~finite if $\lambda<m$. The corresponding $\lambda$-biased random walk on $\T$ is thus recurrent if $\lambda>m$, and transient if $\lambda<m$. For the critical value $\lambda=m$, we know by a subsequent work of Lyons \cite{Lyons2} that the network still has an infinite effective resistance between the root and infinity. More precisely, the critical $\lambda$-biased random walk is null recurrent provided $\sum (k\log k)p_k <\infty$. 

When the edges of $\T$ have random resistances, we are mainly interested in the similar case of critical exponential weighting: to each edge $e$ at depth $d(e)$, we assign the resistance 
\begin{equation}
\label{eq:defi-resis}
r(e) \colonequals m^{d(e)} \xi(e)\,,
\end{equation} 
where, conditionally on $\T$, $\{\xi(e) \}$ are i.i.d.~copies of a nonnegative random variable $\xi$. 
We will call $(\T, \{r(e)\})$ a branching random network of offspring distribution $\mathbf{p}$ and electric resistance $\xi$. 
For convenience, we assume that $(\T, \{r(e)\})$ and $\xi$ are independent and defined under the same probability measure $\P$. 

Let $R_n$ (resp.~$C_n$) be the effective resistance (resp.~effective conductance) between the root $\varnothing$ and the vertices at depth $n$ in $(\T, \{r(e)\})$. 
When $\T$ is a deterministic binary tree, Addario-Berry, Broutin and Lugosi \cite{ABL} showed that as $n\to \infty$,
\begin{equation*}
\E[R_n]= \E[\xi] \,n-\frac{\var[\xi]}{\E[\xi]} \log n +O(1) \quad \mbox{and} \quad \E[C_n]=\frac{1}{\E[\xi]}\,\frac{1}{n} +\frac{\var[\xi]}{\E[\xi]^3}\,\frac{\log n}{n^2}+ O(n^{-2}),
\end{equation*}
provided $\xi$ is bounded away from both zero and infinity. Their arguments are based on the concentration phenomenon of $C_n$ and $R_n$ when the underlying tree is regular. The Efron--Stein inequality is the main tool to deduce the following upper bounds on the variance 
\begin{equation*}
 \var[R_n]=O(1) \quad \mbox{and} \quad \var[C_n]=O(n^{-4}).
\end{equation*} 
A sub-Gaussian tail bound is also established for $R_n$, which gives 
\begin{equation*}
\E \Big[ |R_n-\E[R_n]|^k \Big]=O(1) \qquad \mbox{ for all }k\geq 1. 
\end{equation*}

As observed in the concluding remarks in \cite{ABL}, if the tree $\T$ is random, $C_n$ and $R_n$ are no longer concentrated. 
For any nonnegative random variable $X$, we set $\{X\}\colonequals \frac{X}{\E[X]}$ whenever $0<\E[X]<\infty$. 

\begin{theorem}
\label{thm:cv-Cn}
Assuming that $\E[\xi + \xi^{-1} +\nu^2]<\infty$, we have the almost sure convergence 
\begin{equation}
\label{eq:cv-Cn}
\{C_n\} \build{\longrightarrow}_{n\to\infty}^{} W,
\end{equation}
where $W \colonequals \lim_{n\to \infty} m^{-n}\# \T_n$.
\end{theorem}

We write $W_n\colonequals m^{-n}\# \T_n$. 
When $\E[\nu^2]<\infty$, it is well-known that $(W_n)_{n\geq 1}$ is an $L^2$-bounded martingale. The convergence $W_n \to W$ holds almost surely and in the $L^2$-sense. The limit $W$ is almost surely strictly positive, with 
\begin{equation*}
\E[W]=1 \quad \mbox{and} \quad  \E[W^2]=\frac{\sum k^2 p_k-m}{m(m-1)}.
\end{equation*}
Similarly, for each vertex $x\in \T$, the random variable 
\begin{equation*}
W^{(x)} \colonequals \lim_{n\to \infty} m^{|x|-n}\# \T_n[x]
\end{equation*}
has the same distribution as $W$. Using the tree notation $|x|=n$ to denote a vertex $x$ at depth~$n$, we have $W=m^{-n}\sum_{|x|=n} W^{(x)}$.

Theorem \ref{thm:cv-Cn} answers some questions mentioned at the end of \cite{ABL}. When the offspring number $\nu$ is not deterministic, it implies that the limit distribution of $\{C_n\}$ is absolutely continuous with respect to the Lebesgue measure, which is a ``scaled analogue'' of Question 4.1 in Lyons, Pemantle and Peres \cite{LPP}. For the absolute continuity of $W$, see for instance Theorem 10.4 in Chapter 1 of \cite{AN}.

For our next result, let us define 
\begin{eqnarray}
a_1 &\colonequals& m^{-2}\, \E[\nu(\nu-1)],  \label{eq:a1-constant}  \\
b_1 &\colonequals& \E[\xi], \nonumber \\
c_1 &\colonequals& \frac{a_1 b_1}{1-m^{-1}}. \label{eq:c1-constant}
\end{eqnarray}
Notice that by Theorems 22 and 23 in Dubuc \cite{D}, $\E[W^{-1}]<\infty$ if and only if $p_1\, m <1$. 

\begin{theorem}
\label{thm:Rn}
Assuming that $\E[\xi^2 + \xi^{-1} +\nu^3]<\infty$, we have 
\begin{equation}
\label{eq:cv-ncn}
\lim_{n\to \infty} n \, \E[C_n]=\frac{1}{c_1}.
\end{equation}
If additionally $p_1 \, m < 1$, then 
\begin{equation*}
\lim_{n\to \infty} \frac{\E[R_n]}{n} = c_1\, \E\big[\frac{1}{W}\big].
\end{equation*}
\end{theorem}

If $p_1 \,m \geq 1$, by Fatou's lemma, we deduce from \eqref{eq:cv-Cn} and \eqref{eq:cv-ncn} that 
\begin{equation*}
\liminf_{n \to \infty} \frac{\E[R_n]}{n}=\infty .
\end{equation*} 
See also the remark at the end of Section \ref{sec:bounds}.

To state a more precise asymptotic expansion for $\E[C_n]$, we define 
\begin{eqnarray}
a_2 &\colonequals& m^{-3}\, \E\big[\nu(\nu-1)(\nu-2)\ind{\{\nu \geq 2\}}\big],    \label{eq:a2-constant}  \\
b_2 &\colonequals& \E\big[\xi^2\big], \nonumber \\
c_2 &\colonequals& (1-m^{-2})^{-1} \Big(\frac{3 a_1 ^2}{m -1}+a_2 \Big), \label{eq:c2-constant} \\
c_3 &\colonequals& \frac{2 a_1 c_1}{m-1} -\frac{2 b_1 c_2}{m}, \label{eq:c3-constant} \\
c_4 &\colonequals& \frac{b_1}{1-m^{-1}} \Big( \frac{c_3}{c_1}+a_1\Big)  -b_2 \frac{c_2}{c_1}. \label{eq:c4-constant}
\end{eqnarray}
If $\nu=m\geq 2$ is deterministic, 
\begin{equation*}
c_1 = b_1=\E[\xi], \quad c_2=1, \quad c_3=0 \quad \mbox{ and }\quad  c_4=b_1-\frac{b_2}{b_1}= -\frac{\var[\xi]}{\E[\xi]}.
\end{equation*}

\begin{theorem}
\label{thm:expect-cn}
Assume that $\E[\xi^3 + \xi^{-1} +\nu^4]<\infty$. Then there exists a constant $c_0\in \R$ such that, as $n\to \infty$, 
\begin{equation*}
\E[C_n] = \frac{1}{c_1 n} - \frac{c_4}{c_1^2} \frac{\log n}{n^2} - \frac{c_0}{c_1^2} \frac{1}{n^2} + O(\frac{(\log n)^2}{n^3}).
\end{equation*} 
\end{theorem}

The constant $c_0$ appearing in the expansion above will be defined at the end of Section~\ref{sec:expect-cn}, but its explicit value is unknown to us. 

To further describe the rate of convergence in \eqref{eq:cv-Cn}, we write $\xi_x \colonequals \xi(\{\overleftarrow{x},x\})$ for every vertex $x\neq \varnothing$. Remark that, conditioning on the first $\ell$ levels of the tree $\T$, the random variables $W^{(x)}, |x|=\ell$ are i.i.d.~and independent of $\xi_x, |x|=\ell$. Notice that $W^{(x)} (1 -  \frac{\xi_x}{c_1} \, W^{(x)})$ is of zero mean, because $c_1=\E[\xi]\,\E[W^2]$. When $\E[\xi^2 + \nu^4]<\infty$, one can easily verify that 
\begin{equation*}
\sum_{\ell=1}^\infty \frac{1}{m^\ell} \, \sum_{|x|=\ell}  W^{(x)} \,  \Big(1-\frac{\xi_x}{c_1} \, W^{(x)}\Big)
\end{equation*}
converges in $L^2$. 

\begin{theorem} 
\label{thm:cv-rate}
Assuming that $\E[\xi^3 + \xi^{-1} +\nu^4]<\infty$, we have 
\begin{equation*}
n \big( \{C_n\} - W\big) \, \build{\longrightarrow}_{n \to \infty}^{(\mathrm{P})}\, \sum_{\ell=1}^\infty \frac{1}{m^\ell} \, \sum_{|x|=\ell}  W^{(x)} \,  \Big( 1  -  \frac{\xi_x}{c_1} \, W^{(x)}\Big),
\end{equation*}
and, with the same constant $c_0$ in Theorem \ref{thm:expect-cn}, 
\begin{equation}
\label{eq:rn}
R_n - \Bigg( \frac{c_1}{W} n + \frac{c_4}{W} \log n + \frac{1}{W} \bigg(c_0 -  \frac{1}{W} \sum_{\ell=1}^\infty \frac{1}{m^\ell} \, \sum_{|x|=\ell}  W^{(x)} \,  \Big( c_1  -  \xi_x \, W^{(x)}\Big)\bigg) \Bigg)  \, \build{\longrightarrow}_{n \to \infty}^{(\mathrm{P})}\, 0,
\end{equation}
where $\build{\longrightarrow}_{}^{(\mathrm{P})}$ indicates convergence in probability.
\end{theorem}

The rest of the paper is organized as follows. 
In the next section, we recall Thomson's principle for the effective resistance, and we derive the recurrence relation for $C_n$. In Section \ref{sec:bounds}, we collect some estimates on the moments of $C_n$. The convergence \eqref{eq:cv-ncn} and Theorem \ref{thm:expect-cn} will be shown in Section \ref{sec:expect-cn} by analyzing the recurrence equations on the moments of $C_n$. Similar arguments have already been used in the proof of Theorem 5 in \cite{ABL}. By second moment calculations, we establish Theorems \ref{thm:cv-Cn} and \ref{thm:cv-rate} in Section \ref{sec:as-cv}, and, by proving the uniform integrability of $(n^{-1}R_n)_{n \geq 1}$, we complete the proof of Theorem \ref{thm:Rn} in Section \ref{sec:expected-rn}. Finally, in Section \ref{sec:lambda} we briefly discuss the case when we change the scaling by assigning to each edge $e$ in $\T$ the resistance $\lambda^{d(e)}\xi(e)$ with $\lambda>m$.

\section{Preliminaries}

Consider a general network $G=(V,E)$ with the resistances $\{r(e)\}$. For $x,y\in V$, we write $x \sim y$ to indicate that $\{x,y\}$ belongs to $E$. 
To each edge $e=\{x,y\}$, one may associate two directed edges $\overrightarrow{xy}$ and $\overrightarrow{yx}$. We shall denote by $\overrightarrow{E}$ the set of all directed edges. A flow $\theta$ is a function on $\overrightarrow{E}$ that is antisymmetric, meaning that $\theta(\overrightarrow{xy})=-\theta(\overrightarrow{yx})$. The divergence of $\theta$ at a vertex $x$ is defined by 
\begin{equation*}
\mathrm{div}\, \theta (x) \colonequals \sum_{y\colon y\sim x} \theta(\overrightarrow{xy}).
\end{equation*}
Let $A$ and $Z$ be two disjoint non-empty subsets of $V$: $A$ will represent the source of the network and $Z$ the sink. The flow $\theta$ is from $A$ to $Z$ with strength $\|\theta\|$ if it satisfies Kirchhoff's node law that $\mathrm{div}\, \theta (x)=0$ for all $x\notin A\cup Z$, and that
\begin{equation*}
\|\theta\|=\sum_{a\in A} \sum_{y\sim a, y\notin A} \theta(\overrightarrow{ay})=\sum_{z\in Z} \sum_{y\sim z, y\notin Z} \theta(\overrightarrow{yz}).
\end{equation*}
The effective resistance between $A$ and $Z$ can be defined as 
\begin{equation}
\label{eq:thomson}
R(A \leftrightarrow Z) \colonequals \inf_{\|\theta\|=1} \sum_{e\in E} r(e) \theta(e)^2,
\end{equation}
where the infimum is taken over all flows $\theta$ from $A$ to $Z$ with unit strength. The infimum is always attained at what is called the unit current flow, which satisfies, in addition to the node law, Kirchhoff's cycle law. 
This flow-based formulation of the effective resistance is also called Thomson's principle. 
The effective conductance $C(A \leftrightarrow Z)$ between $A$ and $Z$ is the reciprocal $R(A \leftrightarrow Z)^{-1}$.

Conditionally on the branching random network $(\T,\{r(e)\})$, let $X$ be the associated random walk on the tree $\T$. Let $\omega(x, y), x\sim y$ denote the transition probabilities of $X$, and let $\pi(x), x\in \T$ denote the reversible measure. Writing the conductances $c(e)=1/r(e)$, we have
\begin{equation*}
\pi(x)=\sum_{y\colon y\sim x} c(\{x,y\}) \quad \mbox{and} \quad \omega(x, y)=\frac{c(\{x,y\})}{\pi(x)}.
\end{equation*}
We suppose that the random walk $X$ starts from the vertex $x$ at time 0 under the probability measure $P_{x, \omega}$. 
As a probabilistic interpretation, the effective conductance $C_n \colonequals C(\{\varnothing\}\leftrightarrow \T_n)$ between the root and the level set $\{x\in \T\colon |x|=n\}$ satisfies
\begin{equation*}
C_n = \pi(\varnothing) P_{\varnothing, \omega} \left(\tau_n < T_\varnothing^+ \right),
\end{equation*}
where $\tau_n\colonequals \inf\{k\geq 0\colon  |X_k|=n\}$ and $T^+_\varnothing \colonequals\inf\{k\geq 1\colon X_k=\varnothing \}$.
We see immediately that $C_n\geq C_{n+1}$. 

For $1\leq i\leq \nu$, let $C_{n+1, i}\colonequals C(\{i\} \leftrightarrow \T_{n+1}[i])$ denote the effective conductance between the vertex $i$ and $\T_{n+1}[i]$. We also set $\eta_i\colonequals \xi(\{\varnothing, i\})^{-1}$, $1\leq i\leq \nu$, which are i.i.d., independent of $\nu$. Observe that conditioning on $\nu$, $(C_{n+1, i})_{1\leq i\leq \nu}$ are i.i.d., independent of~$\eta_{i}$, and distributed as $\frac{C_n}{m}$.
Using the series and parallel law of electric networks, we obtain the recurrence relation that for $n\geq 1$, 
\begin{equation}
\label{eq:recur}
C_{n+1}= \sum_{i=1}^\nu \Big( \frac{m}{\eta_i}+\frac{1}{C_{n+1, i}}\Big)^{-1}= \frac1{m} \, \sum_{i=1}^\nu \frac{\eta_i C_{n}^{(i)}}{\eta_i + C_{n}^{(i)}}  \,,
\end{equation}
where for $1\leq i\leq \nu$, $C_{n}^{(i)}\colonequals m\, C_{n+1, i}$ are i.i.d.~copies of $C_{n}$, independent of $(\eta_i)_{1\leq i\leq \nu}$. It is clear that $C_1=m^{-1}\sum_{i=1}^\nu \eta_i$. 
If we set $\xi_i\colonequals \xi(\{\varnothing, i\})= \eta_i^{-1}$ for $1\leq i\leq \nu$, the recurrence equation \eqref{eq:recur} can also be written as 
\begin{equation}
\label{eq:recur-xi}
C_{n+1}= \frac1{m} \, \sum_{i=1}^\nu \frac{C_{n}^{(i)}}{1 + \xi_i C_{n}^{(i)}} \,.
\end{equation}

\section{Bounds on the expected conductance}
\label{sec:bounds}

Let $\eta$ denote the reciprocal $\xi^{-1}$. 

\begin{lemma}
\label{lem:upper-bd-cond}
If $\E[\eta]=\E[\xi^{-1}]<\infty$, then $\E[C_n]\leq \frac{\E[\eta]}{n}$ for all $n\geq 1$. 
\end{lemma}

\begin{proof}
First of all, $\E[C_1]=\E[\eta]$. From \eqref{eq:recur} we obtain for all $n\geq 1$ that
\begin{equation*}
\E[C_{n+1}]=  \E \bigg[ \frac{\eta C_n} {\eta+ C_n} \bigg].
\end{equation*}
By concavity of the function $x\mapsto \frac{x y}{x+y}$, $y>0$ being fixed, 
\begin{equation*}
\E \bigg[ \frac{\eta C_n} {\eta+ C_n}\bigg] \leq \E \bigg[\frac{\eta \E[C_n]}{\eta+ \E[C_n]}\bigg] \leq \frac{\E[\eta]\E[C_n]}{\E[\eta]+ \E[C_n]},
\end{equation*}
It follows that $(\E[C_{n+1}])^{-1} \geq (\E[\eta])^{-1}+ (\E[C_{n}])^{-1} \geq\cdots \geq (n+1) (\E[\eta])^{-1}$. 
\end{proof}

\begin{lemma}
\label{lem:moment-bd-cond}
Assume that $\E[\eta]=\E[\xi^{-1}]<\infty$. For $2\leq k\leq 4$, if $\E[\nu^k]<\infty$, then 
\begin{equation*}
\E[(C_n)^k]=O(n^{-k}) \qquad \mbox{as } n\to \infty. 
\end{equation*}
\end{lemma}

\begin{proof}
Starting from \eqref{eq:recur}, we obtain 
\begin{equation*}
\E\big[(C_{n+1})^2\big]= \frac{1}{m^2} \E[\nu] \, \E \left[\Big(\frac{\eta C_n}{\eta+C_n}\Big)^2\right] + \frac{\E(\nu (\nu-1))}{m^2} \, \big(\E[C_{n+1}]\big)^2,
\end{equation*}
by developing the square and using the independence after conditioning on $\nu$.
Together with Lemma \ref{lem:upper-bd-cond}, it follows that 
\begin{equation*}
\E\big[(C_{n+1})^2 \big]\leq  \frac1{m}\, \E\big[C_n^2\big] + \frac{\E[\nu (\nu-1)]}{m^2} \, \big(\E [C_{n+1}]\big)^2 \leq  \frac1{m}\, \E\big[C^2_n\big] + \frac{\E[\nu (\nu-1)]}{m^2} \frac{(\E[\eta])^2}{(n+1)^2}.
\end{equation*}
Since $m>1$, we get $\E[C_n^2]=O(n^{-2})$ by induction. Furthermore, if $\E[\nu^3]<\infty$, by developing the third power and using the independence, 
\begin{align*}
 \E \big[(C_{n+1})^3\big] & = \E \left[\bigg(\frac1{m} \sum_{i=1}^\nu  \frac{\eta_i C_n^{(i)}}{\eta_i + C_n^{(i)}}\bigg)^3\right] \\
&\leq  \frac{1}{m^2} \E \bigg[\Big(\frac{\eta C_n }{\eta+ C_n}\Big)^3\bigg]   +  \frac{3\E [\nu^2]}{m^3}  \E \bigg[\Big(\frac{\eta  C_n}{\eta+ C_n}\Big)^2\bigg]  \E \bigg[\frac{\eta C_n }{\eta + C_n}\bigg]+   \frac{\E[\nu^3] }{m^3}  \bigg( \E  \bigg[\frac{\eta C_n }{\eta + C_n}\bigg]\bigg)^3 \\
&\leq \frac{1}{m^2} \E\big[C_n^3\big]+ \frac{3\E [\nu^2]}{m^3}  \, \E \big[C_n^2\big]\, \E[C_n]+ \frac{\E[\nu^3]}{m^3} \, (\E[C_n])^3. 
\end{align*}
Thus, $\E[C_n^3 ]=O(n^{-3})$ follows from $\E[C_n]=O(n^{-1})$ and $\E[C_n^2 ]=O(n^{-2})$. The last bound $\E[C_n^4 ]=O(n^{-4})$ is similarly obtained by assuming that $\E[\nu^4]<\infty$. 
\end{proof}

\begin{lemma}
\label{lem:lower-bd-cond}
If $\E[\xi]\in (0,\infty)$ and $\E[\nu^2]<\infty$, then there exists a constant $c>0$ such that $\E[C_n]\geq \frac{c}{n}$ for all $n\geq 1$. 
\end{lemma}

In the following proof, we will use the uniform flow on $\T$ to give an upper bound for $R_n=C_n^{-1}$. Similar arguments can be found in Lemma 2.2 of Pemantle and Peres \cite{PP}.

\begin{proof}
We define on $\T$ the uniform flow $\Theta_{\mathsf{unif}}$ of unit strength (with the source $\{\varnothing\}$) by setting 
\begin{equation*}
\Theta_{\mathsf{unif}}(\{\overleftarrow{x}, x\})=m^{-|x|} \frac{W^{(x)}}{W} \quad \mbox{for every }x\in \T \setminus \{\varnothing\}.
\end{equation*}
According to Thomson's principle \eqref{eq:thomson}, 
\begin{equation}
R_n \leq \sum_{k=1}^n \sum_{|x|=k} m^k \xi_x  \Theta_{\mathsf{unif}}(\{\overleftarrow{x}, x\})^2 
=\sum_{k=1}^n \sum_{|x|=k} m^{-k} \xi_x \Big(\frac{W^{(x)}}{W}\Big)^2. \label{eq:upperRn}
\end{equation}
We write $A\colonequals \sup_{k\geq 1} m^{-k} \# \T_k$, which is square integrable by $L^2$-maximal inequality of Doob. It follows that 
\begin{equation*}
\frac{R_n}{n} \leq \frac{1}{n} \sum_{k=1}^n  \frac{A}{W^2} \bigg(\frac{1}{\# \T_k} \sum_{|x|=k} \xi_x (W^{(x)})^2 \bigg).
\end{equation*}
Using Proposition 2.3 in \cite{PP}, a variant of the strong law of large numbers for exponentially growing blocks of identically distributed random variables being independent inside each block, we have 
\begin{equation*}
\frac{1}{\#\T_k} \sum_{|x|=k} \xi_x (W^{(x)})^2 \build{\longrightarrow}_{k\to\infty}^{\mathrm{a.s.}} \E[\xi]\,\E[W^2].
\end{equation*}
Hence, almost surely
\begin{equation*}
\limsup_{n\to \infty} \frac{R_n}{n} \leq  A\,\E[\xi]\,\frac{\E[W^2]}{W^2},
\end{equation*}
which yields 
\begin{equation*}
\liminf_{n\to \infty} \,n\,C_n \geq  (A\,\E[\xi])^{-1}\frac{W^2}{\E[W^2]},
\end{equation*}
Taking expectation and using Fatou's lemma, we obtain 
\begin{equation*}
\liminf_{n\to \infty} \,n\,\E[C_n] \geq  \frac{\E[W^2 A^{-1}]}{\E[\xi]\,\E[W^2]}>0.
\end{equation*}
The proof is thus completed.
\end{proof}

\rem The Nash-Williams inequality (see Section 2.5 in \cite{LP}) gives the lower bound
\begin{equation*}
R_n \geq \sum_{k=1}^{n} \Big(\sum_{d(e)=k} r(e)^{-1}\Big)^{-1} =
\sum_{k=1}^{n} \Big(\sum_{|x|=k} m^{-k} (\xi_x)^{-1}\Big)^{-1}.
\end{equation*}
Suppose that $\E[\xi^{-1}]<\infty$. 
Proposition 2.3 in \cite{PP} implies that 
\begin{equation*}
\frac{1}{\# \T_k} \sum_{|x|=k} (\xi_x)^{-1} \build{\longrightarrow}_{k\to\infty}^{\mathrm{a.s.}}  \E[\xi^{-1}].
\end{equation*}
With the almost sure convergence $m^{-k}\#\T_k \to W$, it follows that 
\begin{equation*}
\frac{1}{n} \sum_{k=1}^{n} \Big(\sum_{|x|=k} m^{-k} (\xi_x)^{-1}\Big)^{-1} \build{\longrightarrow}_{n\to\infty}^{\mathrm{a.s.}}  \frac{1}{W\E[\xi^{-1}]}.
\end{equation*}
By Fatou's lemma, we obtain 	
\begin{equation*}
\liminf_{n\to \infty} \frac{\E[R_n]}{n} \geq \E\left[\liminf_{n\to \infty} \frac{R_n}{n}\right] \geq \frac{\E[W^{-1}]}{\E[\xi^{-1}]}. 
\end{equation*}
The integrability of $W^{-1}$ is therefore a necessary condition for having $\E[R_n]=O(n)$.

\section{Asymptotic expansion of the expected conductance}
\label{sec:expect-cn}

Within this section, let the assumption $\E[\xi^2+\xi^{-1}+\nu^3]<\infty$ be always in force. We first establish \eqref{eq:cv-ncn} in Theorem \ref{thm:Rn}. Afterwards we will prove Theorem \ref{thm:expect-cn} under the stronger assumption that $\E[\xi^3+\xi^{-1}+\nu^4]<\infty$. 

For every integer $n\geq 1$, we write
\begin{equation*}
x_n \colonequals  \E[C_n], \qquad y_n \colonequals  \E\big[C_n^2\big], \qquad z_n \colonequals  \E\big[C_n^3\big].
\end{equation*}
By Lemma \ref{lem:moment-bd-cond}, we have $x_n=O(n^{-1}), y_n=O(n^{-2})$ and $z_n=O(n^{-3})$.

Observe from \eqref{eq:recur-xi} that $\E[C_{n+1}]= \E \frac{C_n}{1+ \xi\, C_n}$ with $\xi$ and $C_n$ being independent. Then developing the power of $C_{n+1}$, we arrive at 
\begin{equation*}
\E\big[C^2_{n+1}\big] =
\frac{1}{m}  \E \left[\Big(\frac{C_n}{1+\xi \, C_n}\Big)^2\right] + \frac{\E[\nu (\nu-1)]}{m^2} (\E [C_{n+1}])^2 = \frac{1}{m}  \E \left[\Big(\frac{C_n}{1+\xi \, C_n}\Big)^2\right] + a_1 (\E [C_{n+1}])^2
\end{equation*}
and 
\begin{align*}
\E\big[C^3_{n+1}\big]
&=  \frac{1}{m^2} \, \E \left[\Big( \frac{C_n}{1+\xi \, C_n}\Big)^3\right] + \frac{3\E[\nu(\nu-1)]}{m^3} \, \E \left[\Big(\frac{C_n}{1+\xi \, C_n}\Big)^2\right] \, \E  \left[ \frac{C_n}{1+\xi \, C_n} \right] \\
 & \qquad \qquad  +m^{-3}\E\bigg[ \sum_{1\leq i,j,k\leq \nu} \ind{\{i\neq j\neq k\}} \bigg]  \Big(\E\left[\frac{C_n}{1+\xi \, C_n}\right]\Big)^3   \\ 
&=  \frac{1}{m^2} \E \left[\Big( \frac{C_n}{1+\xi \, C_n}\Big)^3\right] + \frac{3a_1}{m} \E \left[\Big(\frac{C_n}{1+\xi \, C_n}\Big)^2\right] \, \E  \left[ \frac{C_n}{1+\xi \, C_n}\right] + a_2 \Big(\E\left[\frac{C_n}{1+\xi \, C_n}\right]\Big)^3,
\end{align*}
with the constants $a_1, a_2$ defined as in \eqref{eq:a1-constant} and \eqref{eq:a2-constant}.

Using the identity $\frac{1}{1+x}=1-x+\frac{x^2}{1+x}$, we obtain 
\begin{eqnarray*}
\E \left[\frac{C_n}{1+ \xi\, C_n}\right]
&=&
\E [C_n] - \E[\xi]\, \E[C_n^2] + \E\left[ \frac{\xi^2 C_n^3}{1+ \xi\, C_n} \right]\\
&=&
\E [C_n] - \E[\xi]\, \E[C_n^2] +  O(n^{-3}),
\end{eqnarray*}
because $\E[C_n^3]=O(n^{-3})$ and $\E[\xi^2]<\infty$. Similarly, 
\begin{equation*}
\E \left[\Big(\frac{C_n}{1+\xi \, C_n}\Big)^2\right] = \E[C_n^2] +O(n^{-3}).
\end{equation*}
Hence, we have 
\begin{eqnarray}
x_{n+1} &=& x_n - b_1 \, y_n + O(n^{-3}), \label{eq:xn0} \\
y_{n+1} &=& \frac{y_n}{m} + a_1 \,  x_{n+1}^2 + O(n^{-3}). \label{eq:yn0} 
\end{eqnarray} 

Remark that 
\begin{equation*}
x_{n+1} = \E\left[\frac{C_n}{1+\xi \, C_n}\right] \geq  \E[C_n] -  \E[\xi]\, \E[C_n^2]= x_n - b_1 y_n.
\end{equation*}
Since $x_n \geq \frac{c}{n}$ by Lemma \ref{lem:lower-bd-cond} and $y_n=O(n^{-2})$, we get $\frac{x_{n}}{x_{n+1}} \leq 1+ \frac{C}{n} $ for some positive constant $C$ independent of $n$. It follows that for any $i < n/2$, 
\begin{equation}
\label{eq:ratio-xn}
1\leq \frac{x_{n-i}}{x_n} \leq \prod_{j=n-i}^{n-1} (1+ \frac{C}{j}) \leq \exp\big(C i/(n-i)\big) \leq 1+ C' \frac{i}{n}
\end{equation}
with another constant $C'>0$.

Still by Lemma \ref{lem:lower-bd-cond}, we can divide all terms in $\eqref{eq:xn0}$ by $x_n x_{n+1}$, which leads to 
\begin{equation}
\label{eq:1/xn-diff}
\frac{1}{x_{n+1}}-  \frac{1}{x_n} = b_1  \frac{y_n}{x_n x_{n+1}} + O(n^{-1}).
\end{equation}
By induction, \eqref{eq:yn0} implies that 
\begin{equation*}
y_{n} = a_1 \, \sum_{i=0}^{n-1} m^{-i} x_{n-i}^2 + O(n^{-3}).
\end{equation*}
Using \eqref{eq:ratio-xn}, we deduce that 
\begin{equation*} 
\frac{y_n}{x_n x_{n+1}} = a_1 \, \sum_{i=0}^\infty m^{-i}+O(n^{-1}) = \frac{a_1}{1- m^{-1}}+O(n^{-1}).
\end{equation*}
It follows from \eqref{eq:1/xn-diff} that
\begin{equation}
\label{eq:1/xn-dif0}
\frac1{x_{n+1}}-  \frac1{x_n}=  \frac{a_1 \, b_1}{1- m^{-1}} + O(n^{-1})= c_1+O(n^{-1}),
\end{equation}
with the constant $c_1$ defined in \eqref{eq:c1-constant}. Consequently, 
\begin{equation}
\label{eq:xn-limit}
\frac{1}{x_n}=c_1 n+O(\log n),
\end{equation}
and
\begin{equation}
\label{eq:xn-limit2}
x_n= \frac{1}{c_1 n } + O(\frac{\log n}{n^2}),
\end{equation}
which gives the convergence \eqref{eq:cv-ncn}.  

\smallskip

Assuming from now on that $\E[\xi^3+\xi^{-1}+\nu^4]<\infty$, we proceed to find higher-order asymptotic expansions for $x_n$. Using the identity $\frac{1}{1+x}=1-x+x^2-\frac{x^3}{1+x}$, we obtain 
\begin{eqnarray*}
\E \left[\frac{C_n}{1+ \xi\, C_n}\right]
&=&
\E [C_n] - \E[\xi]\, \E[C_n^2] + \E[\xi^2]\, \E[C_n^3] - \E\left[ \frac{\xi^3 C_n^4}{1+ \xi\, C_n} \right]\\
&=&
\E [C_n] - \E[\xi]\, \E[C_n^2] + \E[\xi^2]\, \E[C_n^3] + O(n^{-4}),
\end{eqnarray*}
as $\E[\xi^3]<\infty$ and $\E[C_n^4]=O(n^{-4})$ by Lemma \ref{lem:moment-bd-cond}. We prove in the same manner that 
\begin{eqnarray*}
\E \left[\Big(\frac{C_n}{1+\xi \, C_n}\Big)^2\right] &=& \E[C_n^2] -2 \, \E[\xi]\,\E[C_n^3]+O(n^{-4}), \\
\E \left[\Big(\frac{C_n}{1+\xi \, C_n}\Big)^3\right] &=& \E[C_n^3] +O(n^{-4}).
\end{eqnarray*}
Hence, we deduce that
\begin{eqnarray}
x_{n+1} &=& x_n - b_1 y_n + b_2 z_n + O(n^{-4}), \label{eq:xn} \\
y_{n+1} &=& \frac{y_n}{m} + a_1 \,  x_{n+1}^2 -  \frac{2 b_1}{m} z_n + O(n^{-4}) \nonumber \\
&=& \frac{y_n}{m} + a_1\,  x_n^2- \Big( 2 a_1 b_1 x_n y_n + \frac{2 b_1}{m} z_n \Big) + O(n^{-4}), \label{eq:yn} \\
z_{n+1} &=& \frac{z_n}{m^2} + \frac{3a_1}{m} x_{n+1} y_n + a_2 x_{n+1}^3 + O(n^{-4}) 
\nonumber \\ 
&=& \frac{z_n}{m^2} + \frac{3a_1}{m} x_n y_n + a_2 x_n^3 + O(n^{-4}) . \label{eq:zn}
\end{eqnarray} 

Dividing all terms in \eqref{eq:zn} by $x_{n+1}^3$ gives 
\begin{equation*}
\frac{z_{n+1}}{x_{n+1}^3}
= \frac{x_n^3}{x_{n+1}^3} \Big( \frac{1}{m^2} \frac{z_n}{x_n^3} + \frac{3 a_1}{m} \frac{y_n}{x_n^2} + a_2 + O(n^{-1})\Big).
\end{equation*}
Recall that $\frac{x_n}{x_{n+1}}= 1+ O(n^{-1})$ by \eqref{eq:ratio-xn}. Hence, 
\begin{equation*}
\frac{z_{n+1}}{x_{n+1}^3}
= \frac{1}{m^2} \frac{z_n}{x_n^3} + \frac{3 a_1}{m}\frac{y_n}{x_n^2} + a_2 + O(n^{-1}).
\end{equation*}
Since  
\begin{equation}
\label{eq:y-x2}
\frac{y_n}{x_n^2}=\frac{a_1}{1-m^{-1}} +O(n^{-1}),
\end{equation}
we get by induction that 
\begin{equation*}
\frac{z_{n+1}}{x_{n+1}^3}
= \sum_{i=0}^{n} m^{-2 i}\Big( \frac{3 a_1}{m} \frac{a_1}{1-m^{-1}}+a_2\Big) +O(n^{-1}).
\end{equation*}
Then we have 
\begin{equation}
\label{eq:c2-limit}
\frac{z_{n+1}}{x_{n+1}^3} = c_2+O(n^{-1}),
\end{equation}
with the constant $c_2$ defined in \eqref{eq:c2-constant}.

Dividing all terms in \eqref{eq:yn} by $x_{n+1}^2$ gives 
\begin{equation}
\label{eq:y-x2-expan}
 \frac{y_{n+1}}{x_{n+1}^2} = \frac{x_n^2}{x_{n+1}^2} \Big(\frac{y_n}{m x_n^2} + a_1 - 2 a_1 b_1 \frac{y_n}{x_n} -\frac{2 b_1}{m} \frac{z_n}{x_n^2} \Big) + O(n^{-2}).
\end{equation}
For every $n\geq 1$, define 
\begin{eqnarray*}
\varepsilon_n &\colonequals& \frac1{x_{n+1}}- \frac1{x_n}- c_1 , \\
\delta_n &\colonequals& \frac{y_{n+1}}{x_{n+1}^2}-\frac{y_n}{m x_n^2}-a_1. 
\end{eqnarray*}
It has been shown that $\varepsilon_n=O(n^{-1})$. Putting 
\begin{equation*}
\frac{x_n^2}{x_{n+1}^2}= \big(1+ (c_1+\varepsilon_n) x_n\big)^2 = 1+ 2 c_1 x_n+ O(n^{-2})
\end{equation*}
into \eqref{eq:y-x2-expan}, we see that
\begin{equation*}
\delta_n = 2 a_1 c_1 x_n + \Big( \frac{2 c_1}{m} - 2 a_1 b_1 \Big) \frac{y_n}{x_n} - \frac{2 b_1}{m} \frac{z_n}{x_n^2} + O(n^{-2}).
\end{equation*}
By \eqref{eq:y-x2} and \eqref{eq:c2-limit}, it follows that
\begin{equation*}
\frac{\delta_n}{x_n} \build{\longrightarrow}_{n\to\infty}^{} 2 a_1 c_1+\frac{a_1}{1-m^{-1}} \Big(\frac{2 c_1}{m} - 2 a_1b_1 \Big)- \frac{2 b_1 c_2}{m }=c_3,
\end{equation*}
with the constant $c_3$ defined in \eqref{eq:c3-constant}. 
Moreover, in view of \eqref{eq:xn-limit2}, we derive from  
\begin{equation*}
\delta_n = x_n \Big(2 a_1 c_1 + \big(\frac{2 c_1}{m} - 2 a_1 b_1\big) \frac{y_n}{x_n^2} - \frac{2 b_1}{m} \frac{z_n}{x_n^3}\Big) + O(n^{-2})
\end{equation*}
that $\delta_n=\frac{c_3}{c_1}\frac{1}{n} + O(n^{-2} \log n)$. 
If we set 
\begin{equation*}
\Delta_{n+1} \colonequals  \frac{y_{n+1}}{x_{n+1}^2}-  \frac{a_1}{1- m^{-1}},
\end{equation*}
then $\Delta_{n+1}=\frac1{m} \Delta_n + \delta_n$ by the definition of $\delta_n$. It follows by induction that
\begin{equation}
\label{eq:delta}
\Delta_{n+1}=m^{-n}\Delta_1+\sum_{i=0}^{n-1} m^{-i}\delta_{n-i}=\frac{c_3}{c_1 (1- m^{-1})} \frac{1}{n} + O(n^{-2} \log n).
\end{equation}

Going back to \eqref{eq:xn}, we obtain by the definition of $\varepsilon_n$ that 
\begin{eqnarray*}
c_1+\varepsilon_n
&=& 
\frac{1}{x_{n+1}} \, \Big( 1- \frac{x_{n+1}}{x_n}\Big) 
\\
&=&
\frac{x_n}{x_{n+1}} \, \Big( b_1  \frac{y_n}{x_n^2} - b_2 \frac{z_n}{x_n^2} \Big)+ O(n^{-2}) 
\\
&=&
(1+ (c_1 +\varepsilon_n) x_n) \, \Big( b_1  \frac{y_n}{x_n^2} - b_2 \frac{z_n}{x_n^2} \Big)+ O(n^{-2})
\\
&=&
 b_1  \frac{y_n}{x_n^2} +c_1 b_1  \frac{y_n}{x_n} - b_2 \frac{z_n}{x_n^2} + O(n^{-2}). 
\end{eqnarray*}
As $c_1=\frac{a_1 b_1}{1- m^{-1}}$, we deduce that 
\begin{equation}
\label{eq:epsilon}
\varepsilon_n = b_1  \Delta_n + x_n \Big(c_1 b_1  \frac{y_n}{x_n^2} - b_2 \frac{z_n}{x_n^3}\Big)  + O(n^{-2}).
\end{equation}
Using \eqref{eq:xn-limit2}, \eqref{eq:y-x2} and \eqref{eq:c2-limit}, we get that 
\begin{equation*}
\varepsilon_n  = \frac{c_4}{n} + O (n^{-2} \log n), 
\end{equation*}
which implies the absolute convergence of $ \sum_{i=1}^\infty  (\varepsilon_i -\frac{c_4}{i})$.
Hence, 
\begin{equation*}
\frac{1}{x_n} =\frac{1}{x_1} + c_1 (n-1) + \sum_{i=1}^{n-1} \varepsilon_i = c_1 n + c_4 \log n + c_0 + o(1),
\end{equation*} 
with the constant 
\begin{equation*}
c_0 \colonequals -c_1+ \frac{1}{x_1}+\sum_{i=1}^\infty  \big(\varepsilon_i -\frac{c_4}{i} \big)= -c_1+ \frac{1}{\E[\xi^{-1}]}+\sum_{i=1}^\infty  \big(\varepsilon_i -\frac{c_4}{i} \big).
\end{equation*}
Finally we have 
\begin{equation}
\label{eq:expected-cn} 
\E[C_n] = x_n = \frac{1}{c_1 n} - \frac{c_4}{c_1^2} \frac{\log n}{n^2} - \frac{c_0}{c_1^2} \frac{1}{n^2} + O(\frac{(\log n)^2}{n^3}).
\end{equation}

\section{Almost sure convergence and rate of convergence}
\label{sec:as-cv}


To prove Theorems \ref{thm:cv-Cn} and \ref{thm:cv-rate}, let us write 
\begin{eqnarray*}
Y_n &\colonequals&  \{C_n\}- W ,\\
\Pi_n &\colonequals& C_n \Big(\frac{1}{x_{n+1}}- \frac{1}{x_n}  -  \frac{1}{x_{n+1}} \frac{\xi \,C_n}{1+ \xi \, C_n}\Big).
\end{eqnarray*}
For every vertex $x\in \T$ and $j\geq 1$, we also define 
\begin{eqnarray*}
C_j^{(x)} &\colonequals& m^{|x|} \, C \big(\{x\} \leftrightarrow \T_{j+|x|}[x] \big),\\
Y_j^{(x)} &\colonequals& \{C_j^{(x)}\}-W^{(x)} ,\\
\Pi_j^{(x)} &\colonequals& C_j^{(x)} \Big( c_1+ \varepsilon_j - \frac{1}{x_{j+1}} \frac{\xi_x \, C_j^{(x)}}{1+ \xi_x C_j^{(x)}}\Big).
\end{eqnarray*}
Using \eqref{eq:recur-xi}, we have 
\begin{equation*}
\{C_{n}\}
 = \frac{1}{x_{n}}  \frac1{m} \sum_{i=1}^\nu  \frac{C_{n-1}^{(i)}}{1+ \xi_i C_{n-1}^{(i)}}
 = \frac{1}{m} \sum_{i=1}^\nu \{C_{n-1}^{(i)}\} +  \frac{1}{m} \sum_{i=1}^\nu \Pi_{n-1}^{(i)},
\end{equation*}
Using the simple equality $W= m^{-1} \sum_{i=1}^\nu W^{(i)}$, we deduce that 
\begin{equation*}
 Y_n =  \frac{1}{m} \sum_{i=1}^\nu   Y_{n-1}^{(i)} + \frac{1}{m} \sum_{i=1}^\nu \Pi_{n-1}^{(i)}.
\end{equation*}
Since $W= m^{-k} \sum_{|x|=k} W^{(x)}$, by induction,
\begin{equation*}
Y_n = \frac{1}{m^k} \sum_{|x|=k}  Y_{n-k}^{(x)} + \sum_{\ell=1}^k \frac{1}{m^\ell} \, \sum_{|y|=\ell} \Pi_{n-\ell}^{(y)} \quad \mbox{for any }1\leq k <n.
\end{equation*}

\smallskip
\noindent {\it Proof of Theorem \ref{thm:cv-Cn}.} 
Assume that $\E[\xi + \xi^{-1} +\nu^2]<\infty$. Notice that our proof preceding \eqref{eq:ratio-xn} to establish $\frac{x_n}{x_{n+1}}= 1+ O(n^{-1})$ is still valid. Besides, $y_n=\E[C_n^2]=O(n^{-2})$ by Lemma \ref{lem:moment-bd-cond}, and $\frac{y_n}{x_{n+1}}=O(n^{-1})$ by Lemma \ref{lem:lower-bd-cond}. Hence, we derive from the inequality 
\begin{equation*}
\E\big[|\Pi_n|\big]\leq x_n \Big(\frac{1}{x_{n+1}}- \frac{1}{x_n}\Big) +\frac{1}{x_{n+1}} \E \Big[\frac{\xi \, (C_n)^2}{1+ \xi \, C_n}\Big] \leq \frac{x_n}{x_{n+1}}- 1 + \frac{y_n}{x_{n+1}}\E[\xi]
\end{equation*}
that $\E[|\Pi_n|]\leq \frac{C}{n}$ with some constant $C>0$.

Conditioning on the first $k$ levels of the tree $\T$, $(Y_{n-k}^{(x)}, |x|=k)$ are i.i.d.~copies of $Y_{n-k}$. Using the fact that $Y_n$ is of zero mean and uniformly bounded in $L^2$, we can find a constant $C'>0$ such that 
\begin{equation}
\label{eq:Y-L2}
\E\bigg[\Big( \frac{1}{m^k} \sum_{|x|=k}  Y_{n-k}^{(x)}\Big)^2\bigg]
  = m^{-k} \, \E\Big[(  Y_{n-k})^2  \Big]\leq  C'\, m^{-k}.
\end{equation}
Meanwhile, \begin{equation*}
\E \bigg[\sum_{\ell=1}^{k} \frac{1}{m^\ell}\sum_{|y|=\ell} |\Pi^{(y)}_{n-\ell}|\bigg] \leq \sum_{\ell=1}^{k} \frac{C}{n-\ell} \leq \frac{C k}{n-k}.
\end{equation*}
It follows that 
\begin{equation*}
\E \big[| Y_n |\big] \leq  \sqrt{C' \, m^{-k}}+ \frac{C k}{n-k} .
\end{equation*}
By taking $k= C'' \log n$ for some constant $C''$ sufficiently large, we see that 
\begin{equation*}
\E \big[| Y_n |\big] =O(\frac{\log n}{n}).
\end{equation*}
Choose a subsequence $n_j= j^2$. Borel--Cantelli's lemma gives that $Y_{n_j}$ converges to $0$ almost surely. The monotonicity of $C_n$ shows that for any $n_j \leq n < n_{j+1}$, 
\begin{equation*}
\frac{x_{n_{j+1}}}{x_{n_j}} \cdot \{C_{n_{j+1}}\}   \leq  \{C_n\} \leq \frac{x_{n_j}}{x_{n_{j+1}}}\cdot \{C_{n_j}\} .
\end{equation*}
By \eqref{eq:ratio-xn}, the almost sure convergence of $Y_n$ readily follows. $\hfill \Box$

\smallskip
Together with \eqref{eq:xn-limit}, Theorem \ref{thm:cv-Cn} implies that 
\begin{equation}
\label{eq:cv-ncn-as}
n\, C_{n} \build{\longrightarrow}_{n\to\infty}^{\mathrm{a.s.}} \frac{W}{c_1},
\end{equation}
provided $\E[\xi^2+\xi^{-1}+\nu^3]<\infty$.

\smallskip
\noindent {\it Proof of Theorem \ref{thm:cv-rate}.} 
Assume now $\E[\xi^3+\xi^{-1}+\nu^4]<\infty$. First, observe that taking the subsequence $k_n=\frac{4}{\log m}\log n$ in \eqref{eq:Y-L2} yields
\begin{equation*}
n \Big(\frac{1}{m^{k_n}} \sum_{|x|=k_n}  Y_{n-k_n}^{(x)}\Big) \build{\longrightarrow}_{n\to\infty}^{} 0 \quad \mbox{ in $L^2$.}
\end{equation*}
By Borel--Cantelli's lemma, the preceding convergence also holds in the almost sure sense. 
We claim that 
\begin{equation}
\label{eq:clt} 
 \sum_{\ell=1}^{k_n} \frac{1}{m^\ell} \, \sum_{|y|=\ell} n \, \Pi_{n-\ell}^{(y)} 
\, \build{\longrightarrow}_{n\to\infty}^{(\mathrm{P})}\, \sum_{\ell=1}^\infty \frac{1}{m^\ell} \, \sum_{|y|=\ell}  W^{(y)} \,  \Big(1-\frac{\xi_y}{c_1} \, W^{(y)}\Big).  
\end{equation}
In fact, for each vertex $y$ at fixed depth $\ell$, 
\begin{equation*}
n\, C_{n-\ell}^{(y)}\build{\longrightarrow}_{n\to\infty}^{\mathrm{a.s.}} \frac{W^{(y)}}{c_1} \quad \mbox{ and } \quad n \,  \Pi_{n-\ell}^{(y)} \build{\longrightarrow}_{n\to\infty}^{\mathrm{a.s.}} W^{(y)} \Big(1 -  \frac{\xi_y }{c_1} \, W^{(y)}\Big). 
\end{equation*} 
So for any integer $K\geq 1$, 
\begin{equation*}
\sum_{\ell=1}^K \frac{1}{m^\ell} \, \sum_{|y|=\ell} n \, \Pi_{n-\ell}^{(y)} \,\build{\longrightarrow}_{n\to\infty}^{\mathrm{a.s.}} \, \sum_{\ell=1}^K \frac{1}{m^\ell} \, \sum_{|y|=\ell}  W^{(y)} \,  \Big(1-\frac{\xi_y}{c_1} \, W^{(y)}\Big).
\end{equation*} 
Note that  
\begin{equation*}
\E \bigg[\Big(\frac1{m^\ell} \, \sum_{|y|=\ell} n \,\Pi_{n-\ell}^{(y)} \Big)^2\bigg] \leq 
  m^{-\ell}\, n^2 \,\E \big[ \Pi^2_{n-\ell}\big]+ \frac{\E[(\# \T_\ell)^2]}{m^{2\ell}} \,n^2 \big(\E[\Pi_{n-\ell}]\big)^2.
\end{equation*}
On the one hand, 
\begin{eqnarray*}
\E\big[\Pi_n^2 \big] &\leq & 2\Big(\frac{1}{x_{n+1}}- \frac{1}{x_n}\Big)^2 \,  \E\big[C_n^2\big] +\frac{2}{(x_{n+1})^2} \E\Big[\frac{\xi^2\, C_n^4}{(1+\xi C_n)^2}\Big]\\
& \leq & 2\Big(\frac{1}{x_{n+1}}- \frac{1}{x_n}\Big)^2 \,  \E\big[C_n^2\big] +\frac{2}{(x_{n+1})^2} \E\big[\xi^2\big]\, \E\big[C_n^4\big].
\end{eqnarray*}
Using \eqref{eq:1/xn-dif0} and the facts that $x_n$ is of order $n^{-1}$, $\E[C_n^2]= O(n^{-2})$ and $\E[C_n^4]=O(n^{-4})$, we deduce that $\E[\Pi_n^2]=O(n^{-2})$.
On the other hand,  
\begin{eqnarray*}
\E[\Pi_n]
&=&
 \frac{x_n}{x_{n+1}} -1 -  \frac{1}{x_{n+1}} \E[\xi\, C_n^2] + \frac{1}{x_{n+1}}  \E[\xi^2 C_n^3] - \frac{1}{x_{n+1}} \E  \Big[\frac{\xi^3 C_n^4}{1+ \xi C_n} \Big] \\
&=&
\frac{x_n}{x_{n+1}} -1 -  \frac{1}{x_{n+1}} b_1 y_n + \frac{1}{x_{n+1}}  b_2 z_n + O(n^{-3}).
\end{eqnarray*}
It follows by \eqref{eq:xn} that $\E[\Pi_n]=O(n^{-3})$. In particular, $\E[\Pi_{n-\ell}]= O(n^{-3})$ for any $\ell =o(n)$.  
Besides, $m^{-2\ell}\, \E[(\# \T_\ell)^2]$ is uniformly bounded in $\ell$. Hence, there exists some constant $\widetilde C>0$ so that 
\begin{equation*}
\E \bigg[\Big(\frac{1}{m^\ell} \, \sum_{|y|=\ell} n \,\Pi_{n-\ell}^{(y)} \Big)^2 \bigg] 
\leq \widetilde C m^{-\ell} +  \widetilde C n^{-4}  \quad \mbox{ for all } \ell \leq k_n.
\end{equation*}
It follows that    
\begin{equation*}
\lim_{K\to \infty} \limsup_{n\to \infty}  \bigg\|  \sum_{\ell=K}^{k_n}\frac{1}{m^\ell} \, \sum_{|y|=\ell} n \,\Pi_{n-\ell}^{(y)} \bigg \|_{L^2} =\, 0,
\end{equation*}
which yields \eqref{eq:clt}. Therefore, 
\begin{equation*}
n  Y_n = n \Big(\frac{1}{m^{k_n}} \sum_{|x|=k_n}  Y_{n-k_n}^{(x)}\Big) + \sum_{\ell=1}^{k_n} \frac{1}{m^\ell} \, \sum_{|y|=\ell} n \, \Pi_{n-\ell}^{(y)} 
\, \build{\longrightarrow}_{n\to\infty}^{(\mathrm{P})}\, \sum_{\ell=1}^\infty \frac{1}{m^\ell} \, \sum_{|y|=\ell}  W^{(y)} \,  \Big(1-\frac{\xi_y}{c_1} \, W^{(y)}\Big).  
\end{equation*}
In view of \eqref{eq:expected-cn}, we have 
\begin{equation*}
n^2 C_n - \bigg( \frac{W}{c_1}n - \frac{c_4\, W}{c_1^2} \log n  -\frac{c_0 W}{c_1^2} + \frac{1}{c_1}\sum_{\ell=1}^\infty \frac{1}{m^\ell} \, \sum_{|y|=\ell}  W^{(y)} \,  \Big(1-\frac{\xi_y}{c_1} \, W^{(y)}\Big) \bigg)
\, \build{\longrightarrow}_{n\to\infty}^{(\mathrm{P})}\, 0,
\end{equation*}
and the convergence \eqref{eq:rn} follows immediately. $\hfill \Box$

\section{The expected resistance}
\label{sec:expected-rn}

When $\E[\xi^2+\xi^{-1}+\nu^3]<\infty$, it follows from \eqref{eq:cv-ncn-as} that 
\begin{equation*}
\frac{R_n}{n} \build{\longrightarrow}_{n\to\infty}^{\mathrm{a.s.}} \frac{c_1}{W}.
\end{equation*}
The following lemma yields the uniform integrability of $(\frac{R_n}{n}, n\geq 1)$, and completes the proof of Theorem \ref{thm:Rn}. 

\begin{lemma} 
Suppose that $p_1 \,m < 1$ and $\E [\xi^r + \nu^{2 r}] < \infty$ for some $r>1$. Then there exists some $s>1$ such that 
\begin{equation*}
\sup_{n\geq 1} \E\Big[ \Big(\frac{R_n}{n}\Big)^s\Big] < \infty.
\end{equation*} 
\end{lemma}

\begin{proof}
As $p_1 \,m < 1$, by Theorems 22 and 23 in Dubuc \cite{D}, there is some $\alpha>1$ such that 
\begin{equation*}
\E [ W^{-\alpha}] < \infty.
\end{equation*}
In fact, we may take any $\alpha \in (1, -\frac{\log p_1}{\log m})$, with the convention that $-\frac{\log p_1}{\log m}=+\infty$ if $p_1=0$.  
Moreover, $\E[\nu^{2r}]<\infty$ implies that $\E[W^{2r}]<\infty$, according to Bingham and Doney \cite{BD}.

Recall that the martingale $W_k= m^{-k} \#\T_k$ converges in $L^1$ to $W$. Let 
\begin{equation*}
{\mathscr F}_k \colonequals \sigma\{\# \T_i, i\leq k\}, \quad k\geq 0
\end{equation*}
denote the natural filtration associated to $(W_k)_{k\geq 0}$. Since $W_k= \E[W \, |\, {\mathscr F}_k]$,  it follows from Jensen's inequality that $(W_k)^{-\alpha} \le \E[W^{-\alpha} \, |\, {\mathscr F}_k ]$. Consequently, 
\begin{equation}
\label{eq:Wk} 
\sup_{k\geq1} \E \big[(W_k)^{-\alpha} \big] < \infty.
\end{equation}

Fix an arbitrary  $s \in (1, r\wedge \alpha)$. By convexity, we deduce from \eqref{eq:upperRn} that 
\begin{eqnarray*}
\big(\frac{R_n}{n}\big)^s
 &\leq&   \frac1{n} \, \sum_{k=1}^n  \Big( \sum_{|x|=k} m^{-k} \xi_x  \big(\frac{W^{(x)}}{W} \big)^2\Big)^s
  \\
 &\leq&  \frac1{n} \, \sum_{k=1}^n \,  (\# \T_k )^{s-1}   \sum_{|x|=k} \, m^{-k s}\,  (\xi_x)^s\,   \big(\frac{W^{(x)}}{W}\big)^{2s}. 
\end{eqnarray*}

\noindent Since $\E[\xi^s]< \infty$, the proof boils down to showing that 
\begin{equation}
\label{eq:Ls}  
\sup_{k\geq 1} \E \bigg[  (\# \T_k )^{s-1}   \sum_{|x|=k} m^{-k s}    \big(\frac{W^{(x)}}{W}\big)^{2s}\bigg] < \infty.
\end{equation}

Recall that $W= \sum_{|x|=k}  m^{-k} W^{(x)}$, and conditioning on ${\mathscr F}_k$, $(W^{(x)})_{|x|=k}$ are i.i.d.~copies of $W$. Let  $\phi(u)\colonequals - \log \E[e^{- u W}]$ for any $u\geq 0$. Using the elementary identity
\begin{equation*}
a^{-2s} =\frac1{\Gamma(2s)}  \int_0^\infty   \, t^{2 s -1} \, e^{- a t}\, d t  \quad \mbox{for any }a>0,
\end{equation*}  
we get that for any vertex $x$ at depth $k$, 
\begin{eqnarray*}
\E\bigg[ \big(\frac{W^{(x)}}{W}\big)^{2s} \, \Big|\, {\mathscr F}_k \bigg] 
&=&
\frac1{\Gamma(2s)}  \int_0^\infty  \, d t \, t^{2 s -1}  \E\Big[   (W^{(x)})^{2s}\, e^{- t \sum_{|y|=k}  m^{-k} W^{(y)}} \, |\, {\mathscr F}_k \Big] 
\\
&=&
\frac1{\Gamma(2s)}  \int_0^\infty  \, d t \, t^{2 s -1} \, e^{- (\#\T_k-1) \phi(t m^{-k})}\, \E\big[W^{2s} e^{- t m^{-k} W}\big]
\\
&=&
\frac1{\Gamma(2s)} \,m^{2  k s }  \int_0^\infty  \, d u \, u^{2 s -1} \, e^{- (\#\T_k-1) \phi(u)}\, \E\big[W^{2s} e^{- u W}\big].
\end{eqnarray*}
It follows that 
\begin{eqnarray}
I_k &\colonequals & \E \bigg[  (\# \T_k )^{s-1}   \sum_{|x|=k} m^{-k s}    \big(\frac{W^{(x)}}{W}\big)^{2s}\bigg] \nonumber \\
&=& \frac{1}{\Gamma(2s)} \, m^{k s }  \int_0^\infty  \, d u \, u^{2 s -1} \, \E\big[ (\#\T_k)^{s}\,  e^{- (\#\T_k-1) \phi(u)}\big] \, \E\big[W^{2s} e^{- u W}\big]. \label{eq:I_k}
\end{eqnarray}

For any $a>0$, we claim that there exits some positive constant $C=C(a, s)>0$ such that for any $k\ge 1$, 
\begin{equation}
\label{eq:Cas}
m^{k s }  \, \E\big[ (\#\T_k)^{s}\,  e^{-  a (\#\T_k-1) }\big]  \le C.
\end{equation}
Indeed, by discussing whether $\#T_k\geq k^2$ or not, we have
\begin{equation*}
m^{k s }  \, \E\big[ (\#\T_k)^{s}\,  e^{- a \#\T_k   }\big] 
 \leq  m^{k s }  \, \sup_{y \ge k^2} y^s e^{- a y} + m^{k s }  \,  k^{2 s}\,  \P\big( \# \T_k < k^2\big).
\end{equation*}
The first term in the right-hand side is uniformly bounded, while 
\begin{equation*}
m^{k s }  \,  k^{2 s}\,  \P\big( \# \T_k < k^2\big) \leq m^{k s }  \,  k^{2 s+2 \alpha} \, \E\big[ (\# \T_k)^{-\alpha}\big].
\end{equation*}
Note that $ \E[(\# T_k)^{-\alpha}]= O(m^{-\alpha k})$ by \eqref{eq:Wk}. Since $s <\alpha$, we obtain \eqref{eq:Cas}. 

Recall that $\E[W^{2s}]<\infty$ because $s<r$. 
Going back to  the right-hand side of \eqref{eq:I_k}, we split the integral $\int_0^\infty$ into two parts $\int_0^1$ and $\int_1^\infty$. For the part $\int_1^\infty$ we apply \eqref{eq:Cas} with $a=\phi(1)$, and for the part $\int_0^1$ we dominate $\E[W^{2s} e^{- u W}] $ by $ \E[W^{2s}]$, to arrive at
\begin{equation*}
I_k \leq  \frac{C}{\Gamma(2s)} \int_1^\infty  \, d u \, u^{2 s -1} \,   \E\big[W^{2s} e^{- u W}\big]
+  C' m^{k s } \, \int_0^1  \, d u \, u^{2 s -1} \, \E\big[ (\#\T_k)^{s}\,  e^{- \#\T_k \phi(u)}\big]  ,
\end{equation*}
with the finite constant
\begin{equation*}
C'\colonequals \frac{e^{\phi(1)}  \E[W^{2s} ]}{\Gamma(2s)}.
\end{equation*}
Notice that by Fubini's theorem and a change of variables $v=uW$,
\begin{equation*}
\int_1^\infty  d u \, u^{2 s -1} \,   \E\big[W^{2s} e^{- u W}\big] \leq \int_0^\infty  d u \, u^{2 s -1} \,   \E\big[W^{2s} e^{- u W}\big] = \Gamma(2s).
\end{equation*}
To treat the integral from 0 to 1, we remark that  $\lim_{u\to 0} \frac{\phi(u)}{u}= \E[W]=1$. Then there exists some positive constant $c$,  such that $\phi(u) \geq \frac{u}{c} $ for all $0\leq u \leq 1$. It follows that 
\begin{eqnarray*}
I_k &\leq&  C+ C'  m^{k s } \, \int_0^1  \, d u \, u^{2 s -1} \, \E\big[ (\#\T_k)^{s}\,  e^{- \frac{u}{c} \#\T_k  }\big]  \\
&=& C+ C' \, \E \bigg[\int_0^{\#T_k}  d v \, (W_k)^{-s}  v^{2 s- 1} e^{-\frac{v}{c}}\bigg] \\
&\leq& C+C' \, c^{2s} \, \Gamma(2s)\,  \E \big[ (W_k)^{-s} \big].
\end{eqnarray*}
Using again \eqref{eq:Wk} we get that $\sup_{k\geq 1} I_k< \infty$, yielding \eqref{eq:Ls} and completing the proof.
\end{proof}

\section{General exponential weighting}
\label{sec:lambda}

Given the Galton--Watson tree $\T$ and $\lambda>0$, one can do the $\lambda$-exponential weighting of resistance by assigning the resistance $\lambda^{d(e)}\xi(e)$ to each edge $e$ at depth $d(e)$. As before, conditionally on $\T$, $\{\xi(e)\}$ are i.i.d.~positive random variables. In this random electric network, let $C_n(\lambda)$ denote the effective conductance between the root and the vertices at depth $n$.  Instead of \eqref{eq:recur-xi}, the recurrence equation now reads as 
\begin{equation*}
C_{n+1}(\lambda)= \frac1{\lambda} \, \sum_{i=1}^\nu \frac{C_{n}^{(i)}(\lambda)}{1 + \xi_i \, C_{n}^{(i)}(\lambda)},
\end{equation*}
where for $1\leq i \leq \nu$, $C_{n}^{(i)}(\lambda)$ are i.i.d.~copies of $C_{n}(\lambda)$, independent of $(\xi_i)_{1\leq i \leq \nu}$.

\begin{theorem}
\label{thm:lambda}
Fix $\lambda>m$. Assuming that $\E[\xi + \xi^{-1} +\nu^2]<\infty$, we have
\begin{equation*}
\big\{C_n(\lambda)\big\}\build{\longrightarrow}_{n\to\infty}^{\mathrm{a.s.}} W. 
\end{equation*} 
If $\E[\xi^2+\xi^{-1}+\nu^3]<\infty$, then, as $n\to \infty$, the limit of
\begin{equation}
\label{eq:rescale}
\Big(\frac{\lambda}{m}\Big)^n \,\E\big[C_n(\lambda)\big]
\end{equation} 
exists and is strictly positive. 
\end{theorem}

It is easy to see that the limit of the rescaled expected conductance \eqref{eq:rescale} is strictly smaller than $\E[\xi^{-1}]$. However, we are unable to compute it explicitly. 

Basically the proof of Theorem \ref{thm:lambda} goes along the same lines as Theorem \ref{thm:cv-Cn} and that of \eqref{eq:cv-ncn}, except a few minor modifications. We leave the details to the reader.

\section*{Acknowledgements}
We are grateful to an anonymous referee for careful reading of the manuscript and helpful comments.

\end{document}